\documentclass{amsart}
\usepackage[a4paper,
            bindingoffset=0.2in,
            left=1.2in,
            right=1.2in,
            top=1in,
            bottom=1in,
            footskip=.25in]{geometry}
\usepackage[english]{babel}
\usepackage{amsthm}
\usepackage{mathrsfs}
\usepackage{amsmath}
\usepackage{amsfonts}
\usepackage{mathtools}
\usepackage{xcolor}
\usepackage{wrapfig}
\usepackage{hyperref}
 \usepackage{float} 
\usepackage{pgfplots}
\usepackage{cite}
\usepackage{amssymb,enumerate}
\pgfplotsset{compat=1.18,width=10cm}
\usepackage{subcaption}
\usepackage{tikz}
\usepackage{amssymb}
\newtheorem{theorem}{Theorem}
\newtheorem{lemma}[theorem]{Lemma}

\theoremstyle{definition}

\newtheorem{definition}[theorem]{Definition}

\theoremstyle{remark}
\numberwithin{theorem}{section}

\newcommand{\p}[1]{\lvert \nabla#1 \rvert}
\newcommand{\m}[1]{\lVert \nabla#1 \rVert}
\newcommand{\mm}[1]{\lVert #1 \rVert}
\newcommand{\ve}[1]{\lvert #1 \rvert}
\newcommand{\e}{\lambda}
\newcommand{\ep}{\epsilon}

\newcommand{\x}{\rho}
\newcommand{\R}{{\mathbb R}}

\newcommand{\K}{{\mathcal{K}}}
\newcommand{\h}{{\mathcal{T}}}
\newcommand{\Z}{{\mathcal{Z}}}
\newcommand{\op}{{\mathcal{L}}}
\newcommand{\f}{{\mathcal{M}}}
\newcommand{\A}{{\mathcal{A}_*}}
\newcommand{\q}{\Omega}
\newcommand{\ba}{\beta}
\newcommand{\ga}{\gamma}
\newcommand{\al}{\alpha}
\newcommand{\be}{\begin{equation}}
\newcommand{\ee}{\end{equation}}
\numberwithin{equation}{section}
\allowdisplaybreaks
\author[Anupma Arora]
{Anupma Arora}
\address{Anupma Arora\hfill\break
Department of Mathematics\newline
 Birla Institute of Technology and Science Pilani \newline
  Pilani Campus, Vidya Vihar \newline
 Pilani, Jhunjhunu \newline
 Rajasthan, India - 333031}
 \email{p20200439@pilani.bits-pilani.ac.in;anupmaarora17@gmail.com }
 
 \author[Gaurav Dwivedi ]
{Gaurav Dwivedi}
\address{Gaurav Dwivedi \hfill\break
Department of Mathematics\newline
 Birla Institute of Technology and Science Pilani \newline
  Pilani Campus, Vidya Vihar \newline
 Pilani, Jhunjhunu \newline
 Rajasthan, India - 333031}
 \email{gaurav.dwivedi@pilani.bits-pilani.ac.in}
\begin{document}
\title[ Kirchhoff type multiphase problem ] {On Choquard-Kirchhoff type critical multiphase problem}

\begin{abstract}
  In this paper, we obtain the existence of weak solutions to the Choquard-Kirchhoff type critical multiphase problem:
\begin{equation*}
\left\{\begin{array}{cc}
    &-M(\varphi_{\h}(\p{u}))div(\p{u}^{p(x)-2}\nabla u+a_1(x)\p{u}^{q(x)-2}\nabla u+a_2(x)\p{u}^{r(x)-2}\nabla u)\\
       & =\e g(x)\ve{u}^{\gamma(x)-2}u+\theta B(x,u)+\kappa \left(\int_{\q}\frac{F(y,u(y))}{\ve{x-y}^{d(x,y)}}\, dy\right) f(x,u) \ \text{in} \ \q,\\
       & u=0  \ \text{on} \ {\partial \Omega}.
\end{array}\right.
\end{equation*}

The term $B(x,u)$ on the right-hand side generalizes the critical growth. We obtain existence and multiplicity results by establishing certain embedding results and concentration compactness principle along with the Hardy-Littlewood-Sobolev type inequality for the Musielak Orlicz Sobolev space $ W^{1,\h}(\q)$.
\end{abstract} 
\maketitle
%\begin{document}

\maketitle
\section{Introduction}
This article investigates the existence of weak solutions for the following problem:
\begin{equation}
\label{model}
\tag{$P_{\e}$}\left\{
\begin{array}{cc}
    &-M(\varphi_{\h}(\p{u} ))div(\p{u}^{p(x)-2}\nabla u+a_1(x)\p{u}^{q(x)-2}\nabla u+a_2(x)\p{u}^{r(x)-2}\nabla u)\\
        &=\e g(x)\ve{u}^{\gamma(x)-2}u+\theta B(x,u)+\kappa \left(\int_{\q}\frac{F(y,u(y))}{\ve{x-y}^{d(x,y)}}\, dy\right) f(x,u) \ \text{in} \ \q,\\
        & u=0  \ \text{on} \ {\partial \Omega},
\end{array}\right.
\end{equation}
where $\q$ is a bounded domain with a smooth boundary $\partial \Omega$ and 
\[
    \varphi_{\h}(u)=\int_{\q}\left(\frac{| u|^{p(x)}}{p(x)}+a_1(x)\frac{| u|^{q(x)}}{q(x)}+a_2(x)\frac{| u|^{r(x)}}{r(x)}\right)dx,\]
\[B(x,t)=c_1(x) \ve{t}^{\ba_1(x)-2}t+c_2(x)a_1(x)^{\frac{\ba_2(x)}{q(x)}}\ve{t}^{\ba_2(x)-2}t+c_3(x)a_2(x)^{\frac{\ba_3(x)}{r(x)}}\ve{t}^{\ba_3(x)-2}t,\]
and $\e,\theta,\kappa$ are positive parameters. Additionally, $a_1(\cdot),a_2(\cdot)\in L^{\infty}(\q)$ with $a_1(x),a_2(x)\geq 0$ a.e. in $\Omega,$ and the exponents $p,q,r,\ga, \ba_1,\ba_2,\ba_3 \in C_+(\overline{\q}).$ The function $f:\q \times \R \to \R$ is a continuous function with primitive $F,$ and $ d:\overline{\q} \times \overline{\q} \to \R$ is a continuous function such that $0 < d^{-} \leq d^{+} < N $. Moreover, we assume that the following conditions are satisfied:
 \begin{itemize}
    \item [$(M_1)$] The Kirchhoff term $M\in C(\R^+, \R^+)$ and there exists $m_0>0$ such that $M(t)\geq m_0$ for all $t\geq 0$.
     \item [$(M_2)$]There exists $\eta >\frac{r^+}{\ba_1^-}$ such that $\Hat{M}(t)\geq \eta t M(t)$ for all $t \geq 0,$ where $\Hat{M}(t)=\int_{0}^{t}M(s)\,ds$.
     \item[$(M_3)$] There exists $\zeta>\frac{\ba_3^+}{p^-}$ such that $\limsup_{t \to 0}\frac{\Hat{M}(t)}{t^{\zeta}}<\infty$.
      \item [$(H_1)$] $p(x)<N$ and $\ga(x)<p(x)<q(x)<r(x)<p^*(x)$ for all $x\in \overline{\Omega}.$
      \item[$(H_2)$]$ p,q,r\in C^{0,1}(\overline{\q}),\ p(x)<q(x)<r(x)<N$ and $0\leq a_1,a_2\in C^{0,1}(\overline{\q})$. Moreover, $\frac{r^+}{p^-}<1+\frac{1}{N}\cdot$
     \item[$(H_3)$]$0<g\in L^{\xi(x)}(\q),$ where $\xi(x)\in C_+(\overline{\q})$ is such that $\ga(x)<\frac{\xi(x)-1}{\xi(x)}p^*(x)$ for all $x\in \overline{\q}$.
     \item[$(H_4)$] The functions $c_i \in L^{\infty}(\q)$ for $i=1,2,3$ are such that $c_1>0, c_2,c_3\geq 0; \ r(x)<\ba_1(x)\leq p^*(x) \ \forall \  x\in \overline{\q};$
\[p^*(x)-\ba_1(x)=q^*(x)-\ba_2(x)=r^*(x)-\ba_3(x)  \ \forall \  x\in \overline{\q} \]
and $C=\{x \in \q: p^*(x)=\ba_1(x)\}\neq \emptyset$.
    \item [$(L_1)$] For all $x \in \overline{\q}$ and $t \in \R$ there exists a constant $c>0$ and $\al_1,\al_2,\al_3 \in C_+(\overline{\q})$ such that 
    \[\ve{f(x,t)}\leq c\left(\ve{t}^{\al_1(x)-1}+a_1(x)^{\frac{\al_2(x)}{q(x)}}\ve{t}^{\al_2(x)-1}+a_2(x)^{\frac{\al_3(x)}{r(x)}}\ve{t}^{\al_3(x)-1}\right)\]
    where $\al_i \in M_i(\q)$ for $i=1,2,3$, $1<\al_1(\cdot)<\al_2(\cdot)<\al_3(\cdot)$ and $\al_1^->\frac{r^+}{2}$ with
    \begin{align*}
        M_1(\q)&=\left\{ s\in C_+(\overline{\q}): 1<s(x)\frac{2N}{2N-d^-}\leq s(x)\frac{2N}{2N-d^+}<p^*(x) \right\},\\
        M_2(\q)&=\left\{ s\in C_+(\overline{\q}): 1<s(x)\frac{2N}{2N-d^-}\leq s(x)\frac{2N}{2N-d^+}<q^*(x) \right\},\\
        M_3(\q)&=\left\{ s\in C_+(\overline{\q}): 1<s(x)\frac{2N}{2N-d^-}\leq s(x)\frac{2N}{2N-d^+}<r^*(x) \right\}.
    \end{align*}
   \item [$(L_2)$] There exists $l$ with $\frac{r^+}{\eta}<l<\ga^-$ such that $0<lF(x,t)\leq 2 f(x,t)t$ holds for all $t\in \R^+$ where $\eta$ is as defined in $(M_2)$ and $F(x,t)=\int_{0}^{t}f(x,s)\,ds \ \forall t>0$.
   \item [$(L_3)$] $f(x,-t)=-f(x,t),\, \forall (x,t) \in \Omega \times \R$.
  \end{itemize}

The operator involved in \eqref{model} is known as multiphase operator with variable exponents. To our knowledge, only two research works deal with multiphase operators with variable exponents. Dai and Vetro \cite{mp1} introduced this operator, analyzed the associated functional spaces, specifically the Musielak-Orlicz-Sobolev space $W^{1,\h}(\q)$ and established embedding and regularity results. They also considered a problem involving a convection term. Subsequently, Vetro \cite{mp2} derived a priori upper bounds and demonstrated the existence of positive and negative weak solutions. For multiphase problems with constant exponents, we refer to \cite{mp3}.
Notably, this operator generalizes the double-phase operator with variable exponents, which is recovered when $a_2 \equiv 0$.

The main goal of this paper is twofold. Firstly, we will establish certain embedding results for the space $W^{1,\h}(\q)$. Inspired by Crespo Blanco et al. \cite{vdp2}, we prove that $W^{1,\h}(\q)\hookrightarrow L^{\h_*}(\q),$ where $\h_*$ is Sobolev conjugate of function $\h$. Then, following the idea of \cite{vdp6}, we establish the embedding 
\[W^{1,\h}(\q)\hookrightarrow L^{\A}(\q),\]
where $\A(x,t)=t^{p^*(x)}+a_1(x)^{\frac{q^*(x)}{q(x)}}t^{q^*(x)}+a_2(x)^{\frac{r^*(x)}{r(x)}}t^{r^*(x)}$ for $x\in \overline{\q}$ and $t\in[0,\infty)$.
Our problem \eqref{model} has critical growth. The main difficulty in dealing with critical problems is the lack of compactness. The non-compactness associated with the variational approach adds complexity and interest to the critical problems. To address the issue of lack of compactness, Lions \cite{Lions1, Lions2} introduced the two concentration-compactness principles \cite[Lemma I.1]{Lions1} and \cite[ Lemma I.1]{Lions2}. Following the foundational work of Brezis and Nirenberg \cite{Brezis} on the Laplace equation, these principles have been extensively expanded and generalized in various directions. For $p(x)$-Laplacian we refer interested readers to \cite{Silva,cp1,cp2,cp3,cp4,cp5} and references therein. 
For double-phase problems involving critical growth $p^*$, one can refer to \cite{dpk1,cdp3}. Hai Ha and Ho \cite{vdp7} proved multiplicity results for double phase problems with critical growth switching between $p^*(\cdot)$ and $q^*(\cdot)$ in the bounded domain. They also developed the concentration-compactness principle (CCP) for the corresponding Musielak-Orlicz-Sobolev space. Additional works addressing such growth include \cite{cdp1, cdp2}.
 To address the critical nonlinearity $B(x, u)$ in our context, we establish the concentration-compactness principle for the space $W_0^{1,\h}(\q).$

The problem \eqref{model} has Choquard type nonlinearity, which is primarily tackled by the Hardy-Littlewood-Sobolev inequality (HLSI). In the second part of this paper, we establish an HLSI-type inequality for the space $W^{1,\h}(\q)$ and derive the existence and multiplicity results for a critical Kirchhoff type multiphase problem with generalized Choquard nonlinearity. Lieb-Loss proved the HLSI for the Sobolev space \cite{Lieb}. Alves and Tavaras \cite{cv1} established the HLSI for variable exponent Sobolev spaces to deal with the Choquard problem involving variable exponents. Alves et al. \cite{co} dealt with the generalized Choquard problem in Orlicz spaces. Arora et al. \cite{cn3} explored the Choquard problem involving the double-phase operator. Gupta and Dwivedi \cite{Gupta} worked on a generalized Choquard problem with vanishing potential in the fractional Musielak-Orlicz space. Further, for existence results on Choquard problems, we refer to \cite{cn1,cn2}(for Laplacian), \cite{fr1}(for fractional Laplacian),\cite{cv3,cv4,cv5}(for $p(x)$ Laplacian), \cite{cv2} (for fractional $p(x)$ Laplacian), \cite{cpq}(for $p-q$ Laplacian) and references therein. For a survey of Choquard problems, see \cite{guide,guide2}.

To our knowledge, this is the first work dealing with critical and Choquard nonlinearities for multiphase operators. Moreover, our problem contains a nonlocal Kirchhoff term. The Kirchhoff-type problem originates from the work of \cite{kirchhoff1876}. For further details on Kirchhoff-type problems, one can see \cite{dpk3,dpk1, FJSA,hamydy2011existence,zhao} and references therein. 

Our methodology to establish the existence and multiplicity results for \eqref{model} is inspired by \cite{cv3}. Next, we state our main results:
  \begin{theorem}\label{thm1}
   Let $(H_2)-(H_4),$ $(M_1)-(M_3),$ and $ (L_1)-(L_2)$ hold. Then there exists $\e_0, \kappa^*>0$ such that for every $0<\e<\e_0$ and $\kappa >\kappa^*,$ the problem \eqref{model} admits at least two nontrivial solutions. 
  \end{theorem}
  
  \begin{theorem}\label{thm2}
  Let $(H_2)-(H_4), \,(L_1)-(L_3)$ and $(M_1)-(M_2)$ hold. Then there exists $\e_0$ and a sequence $\{\theta_n\} \in (0,\infty)$ with $\theta_{n+1}<\theta_{n}$ such that the problem \eqref{model} has at least $n$ pair of solutions for $\theta_{n+1}<\theta\leq \theta_{n}$ and $0<\e<\e_0$. 
  \end{theorem}
This paper is organized as follows. Section 2 focuses on proving the main tools, which include embedding results (Lemma \ref{lem1}, Lemma \ref{emb}, and Lemma \ref{embb}), the concentration compactness principle for the space \( W_0^{1,\h}(\Omega) \) (Theorem \ref{thm3}), and a variant of the HLSI for the space \( W_0^{1,\h}(\Omega) \) (Lemma \ref{HL}). Section 3 is dedicated to proving Theorem \ref{thm1}. Finally, the proof of Theorem \ref{thm2} is presented in Section 4.
\section{Functional Space setup}
In this section, we present some preliminary definitions and results. For a detailed study of Musielak-Orlicz spaces, we refer to \cite{Diening, Haito}. Throughout this article, we assume that the $(H_1)$ condition is satisfied.

Define $\psi:\q \times [0,\ \infty) \to  [0,\ \infty)$ as \[\psi(x,t)=b_1(x)t^{s_1(x)}+b_2(x)t^{s_2(x)}+b_3(x)t^{s_3(x)},\]
where $s_i \in C_+(\overline{\q})$ are such that $s_1(\cdot)<s_2(\cdot)<s_3(\cdot)$, $0<b_1(\cdot)\in L^{1}(\q)$ and $0\leq b_2(\cdot),b_3(\cdot)\in L^{1}(\q)$.
The function $\psi(x,t)$ is a generalised $N$ function and satisfying $\Delta_2$ condition
\[\psi(x,2t)\leq 2^{s_3^+} \psi(x,t).\]
The Musielak Lebesgue space $L^{\psi}(\q)$ is defined as
\[L^{\psi}(\q):=\left\{u: \q \to \R \ \text{is measurable function}\ : \int_{\q}\psi(x,\ve{u})dx< \infty \right\},\]
with the norm 
\[\mm{u}_{\psi}:=\inf\left\{\Lambda>0: \int_{\q}\psi\left(x,\left\lvert\frac{u}{\Lambda} \right\rvert\right)dx \leq 1  \right\}.\]
The space $L^{\psi}(\Omega)$ is a Banach space. The following lemma follows from \cite[Proposition 3.2]{mp1}.
\begin{lemma}\label{rel}
Let $\x_{\psi}(z)=\int_{\q}\psi(x,\ve{z})\,dx$. Then $\mm{z}_{\psi}$ and $\x_{\psi}(z)$ are related as follows:
\begin{itemize}
    \item [(i)] For $z \neq 0, \  \mm{z}_{\psi}=c \iff \x_{\psi}(z/c)=1,$
       \item[(ii)]$\mm{z}_{\psi}<1(=1,>1) \iff \rho_{\psi}(z)<1(=1,>1),$
        \item[(iii)]$\mm{z}_{\psi}<1,$ then $\mm{z}_{\psi}^{s_3^+}\leq \rho_{\psi}(z) \leq \mm{z}_{\psi}^{s_1^-},$
        \item[(iv)]$\mm{z}_{\psi}>1,$ then $\mm{z}_{\psi}^{s_1^-}\leq \rho_{\psi}(z) \leq \mm{z}_{\psi}^{s_3^+},$
        \item[(v)]$\mm{z}_{\psi} \to 0 \iff \rho_{\psi}(z) \to 0.$
\end{itemize}
\end{lemma}
If $b_2,b_3\equiv 0$ and $b_1\equiv 1$, then the space $L^{\psi}(\q)$ reduces to the variable Lebesgue space $L^{p(x)}(\q)$. The space $L^{p(x)}$ endowed with the norm $\mm{\cdot}_{p(\cdot)}$. If $u\in L^{p(\cdot)}(\q)$ and  $v\in L^{p'(\cdot)}(\q),$ then $uv\in L^1(\q)$ and the generalised H\"older's inequality is given by  
\begin{equation}\label{25}
    \left|\int_{\q} u(x)v(x)dx\right|\leq \left(\frac{1}{p^-}+\frac{1}{(p')^-} \right) \mm{u}_{p(x)}\mm{v}_{p'(x)}\leq 2\mm{u}_{p(x)}\mm{v}_{p'(x)}, 
\end{equation}
 where $\frac{1}{p(x)}+\frac{1}{p'(x)}=1.$

The function $\h: \q \times [0 , \infty) \to  [0 , \infty)$ defined by $\h(x,t)=t^{p(x)}+a_1(x)t^{q(x)}+a_2(x)t^{r(x)},$ is a generalised $N$ function and satisfying $\Delta_2$ condition as follows:
\[\h(x,2t)\leq 2^{r^+} \h(x,t).\]
\begin{definition} \cite{mp1}
The Musielak-Lebesgue space $L^{\h}(\q)$ is defined as
 $$L^{\h}(\q):=\left\{u: \q \to \R \ \text{measurable function}\ : \int_{\q}\h(x,\ve{u})dx< \infty \right\}, $$
  with the norm 
 $$\mm{u}_{\h}:=\inf\left\{\Lambda>0: \int_{\q}\h\left(x,\left\lvert\frac{u}{\Lambda} \right\rvert\right)dx \leq 1  \right\}. $$    
\end{definition}
Next, we define the Musielak-Orlicz-Sobolev space $W^{1,\h}(\q)$
\[ W^{1,\h}(\q):=\{u\in L^{\h}(\q): \p{u} \in L^{\h}(\q)\}, \]
endowed with the norm 
\[\mm{u}_{1,\h}:=\mm{u}_{\h}+\m{u}_{\h}.\]
 $W_{0}^{1,\h}(\q)$ is defined as the completion of $C_c^{\infty}(\q).$
Moreover, if $(H_1)$ holds then there exists a constant $c>0$ such that $\mm{u}_{\h}\leq c \m{u}_{\h}\ \forall u\in W_{0}^{1,\h}(\q)$. So, we can equip the space $W_{0}^{1,\h}(\q)$ with the equivalent norm $\mm{u}=\m{u}_{\h}$.
  The spaces $L^{\h}(\q), W^{1,\h}(\q)$ and $W_0^{1,\h}(\q)$ are reflexive Banach spaces \cite[Proposition 3.1]{mp1}. Throughout this article, `$\hookrightarrow$' denotes continuous embedding, and `$\hookrightarrow \hookrightarrow$' denotes a compact embedding.
\begin{lemma} \cite[Proposition 3.3]{mp1}\label{em}
    The following embeddings hold:
    \begin{enumerate}[(i)]
        \item $W^{1,\h}(\q) \hookrightarrow L^{s(\cdot)}(\q)$ and $W_{0}^{1,\h}(\q) \hookrightarrow L^{s(\cdot)}(\q)$ for $s \in C(\overline{\q})$ with $1\leq s(x)\leq p^*(x)$ for $x\in \overline{\q}$.
        \item $W^{1,\h}(\q) \hookrightarrow \hookrightarrow L^{s(\cdot)}(\q)$ and $W_{0}^{1,\h}(\q) \hookrightarrow\hookrightarrow L^{s(\cdot)}(\q)$ for $s \in C(\overline{\q})$ with $1\leq s(x)<p^*(x)$ for $x\in \overline{\q}$.
        \end{enumerate}
\end{lemma}
\begin{lemma} \label{g}
If $(H_1),(H_3)$ hold, then $W^{1,\h}(\q)\hookrightarrow\hookrightarrow L_{g(x)}^{\ga(x)}(\q)$.
\end{lemma}
The proof of this above lemma follows from Lemma \ref{em} and \cite[Theorem 2.7]{Nehari}.
\begin{lemma}\cite[Proposition 4.5]{mp1}\label{S}
Let $\op: W^{1,\h}(\q) \to W^{1,\h}(\q)^*$ be defined as
\be
\langle\op(u),v\rangle=\int_{\q}\left(\p{u}^{p(x)-2}+a_1(x)\p{u}^{q(x)-2}+a_2(x)\p{u}^{r(x)-2}\right)\nabla u\nabla v \,dx.
\ee
Then  $\op$ is bounded, continuous, and of $S_+$ type.
\end{lemma}
 \begin{definition} 
Let $\mathcal{T}^{-1}(x, \cdot) : [0, \infty) \to [0, \infty)$ for all $x \in \overline{\Omega}$ be the inverse function of $\mathcal{T}(x, \cdot)$. Define $\mathcal{T}_{*}^{-1} : \Omega \times [0, \infty) \to [0, \infty)$ by 
\begin{equation} 
\mathcal{T}_{*}^{-1}(x,s) = \int_0^s \frac{\mathcal{T}^{-1} \left( x, \tau\right)}{\tau^{\frac{N+1}{N}} } \, d\tau \quad \text{for all } (x,s) \in \overline{\Omega} \times [0, \infty),
\end{equation}
where $\mathcal{T}_{*} :(x,t) \in \overline{\Omega} \times [0, \infty) \mapsto s \in [0, \infty)$ is such that $\mathcal{T}_{*}^{-1}(x,s) = t$. The function $\mathcal{T}_{*}$ is called the Sobolev conjugate function of $\mathcal{T}$.
\end{definition} 
\begin{lemma} \label{lem1}
Let $(H_2)$ hold. Then 
 \begin{enumerate}[(i)]
 \item The embedding $W^{1,\h}(\q)\hookrightarrow L^{\h_*}(\q)$ holds.
 \item For $\psi:\q \times [0 , \infty) \to  [0 , \infty)$ with $\psi \in N(\q)$ and $\psi<<\h_*,$  $W^{1,\h}(\q)\hookrightarrow\hookrightarrow L^{\psi}(\q)$ holds.
 \end{enumerate}
\end{lemma}
\begin{proof}The proof follows from \cite[Theorem 1.1, Theorem 1.2]{Musielak}.
It is enough to show that the following equations \eqref{2.2} and \eqref{2.3} hold. 
\be \label{2.2}
\lim_{t \to \infty} \h_*^{-1}(x,t)=\infty \ \forall x\in \q.
\ee
There exists $\delta <\frac{1}{N}, C, t_0$ such that 
\be \label{2.3}
\left \lvert \frac{\partial \h(x,t)}{\partial x_j}\right \rvert \leq C(\h(x,t))^{1+\delta} \forall x\in \q \ \text{and}\ t\geq t_0,
\ee
 whenever $\nabla a_1(x), \nabla a_2(x), \nabla p(x),\nabla q(x)$ and $\nabla r(x)$ exist.

Firstly, let us prove \eqref{2.2}. For $t>1$, we have 
\begin{align*}
  \h(x,t) &\leq \left(1+\mm{a_1}_{\infty}+\mm{a_2}_{\infty}\right)t^{r^+}\\
  \h^{-1}(x,t)&\geq \left(1+\mm{a_1}_{\infty}+\mm{a_2}_{\infty}\right)^{\frac{-1}{r^+}}t^{\frac{1}{r^+}}\\
  \h_*^{-1}(x,t)&\geq \int_{1}^{t}\frac{\h^{-1}(x,s)}{{s}^{\frac{N+1}{N}}}\,ds\\
  &\geq \left(1+\mm{a_1}_{\infty}+\mm{a_2}_{\infty}\right)^{\frac{-1}{r^+}}\int_{1}^{t} s^{\frac{1}{r^+}-\frac{N+1}{N}}\,ds\\
  &= c\left(t^{\frac{1}{r^+}-\frac{1}{N}}-1\right) \to \infty  \ \text{as} \ t \to \infty.
\end{align*}
Next, we claim \eqref{2.3}. Let $c_s$ be the Lipschitz coefficient for the Lipschitz continuous function $s$. 
From $\frac{r^+}{p^-}<1+\frac{1}{N},$ there exists $\zeta>0$ small enough such that 
\[
\frac{r^++\zeta}{p^-}<1+\frac{1}{N}
\]
and $\ln{t}\leq c t^{\zeta}$. For $t\geq 1$ we have 
\begin{align*}
   \left \lvert \frac{\partial \h(x,t)}{\partial x_j}\right \rvert  &\leq c_p t^{p(x)} \ln{t}+a_1(x)c_q t^{q(x)} \ln{t}+c_{a_1} t^{q(x)}+a_2(x)c_r t^{r(x)} \ln{t}+c_{a_2} t^{r(x)}\\
   &\leq  c_p c t^{p(x)+\zeta} +a_1(x)c_q c t^{q(x)+\zeta} +c_{a_1} t^{q(x)}+a_2(x)c_r c t^{r(x)+\zeta}+c_{a_2} t^{r(x)}\\
   &\leq \left(c_p c+\mm{a_1}_{\infty}c c_q+c_{a_1}+\mm{a_2}_{\infty}c c_r+c_{a_2}\right)t^{r(x)+\zeta}\\
   &\leq \left(c_p c+\mm{a_1}_{\infty}c c_q+c_{a_1}+\mm{a_2}_{\infty}c c_r+c_{a_2}\right) \left(\h(x,t)\right)^{\frac{r^+ +\zeta}{p^-}}
\end{align*}
Thus, \eqref{2.3} holds with $C=c_p c+\mm{a_1}_{\infty}c c_q+c_{a_1}+\mm{a_2}_{\infty}c c_r+c_{a_2}$, $t_0=1$ and $\delta=\frac{r^+ +\zeta}{p^-}-1$.
\end{proof}
The proof of the following two lemmas follows from \cite[Proposition 3.4, Proposition 3.7]{vdp6}.
\begin{lemma}\label{emb}
The embedding $W^{1,\h}(\q)\hookrightarrow L^{\A}(\q)$ holds, where 
$\A(x,t)=t^{p^*(x)}+a_1(x)^{\frac{q^*(x)}{q(x)}}t^{q^*(x)}+a_2(x))^{\frac{r^*(x)}{r(x)}}t^{r^*(x)}$ for $x\in \overline{\q}$ and $t\in[0,\infty)$.
\end{lemma}
\begin{lemma}\label{embb}
    Define $\Z(x,t)=t^{\al_1(x)}+a_1(x)^{\frac{\al_2(x)}{q(x)}}t^{\al_2(x)}+a_2(x))^{\frac{\al_3(x)}{r(x)}}t^{\al_3(x)}\ \forall (x,t)\in \overline{\q} \times [0,\infty)$
    where $ \al_1,\al_2,\al_3 \in C_+(\overline{\q})$ are such that  $\al_1\leq p^*(x),\ \al_2(x)\leq q^*(x)$ and  $\al_3(x)\leq r^*(x)$ for $x\in \overline{\q}$. Then $W^{1,\h}(\q)\hookrightarrow L^{\Z}(\q)$. Moreover, the embedding is compact provided  $\al_1< p^*(x),\ \al_2(x)< q^*(x)$ and  $\al_3(x)< r^*(x)$. 
\end{lemma}
Let $\f(x,t):=c_1(x)\ve{t}^{\ba_1(x)}+c_2(x)a_1(x)^{\frac{\ba_2(x)}{q(x)}}\ve{t}^{\ba_2(x)}+c_3(x)a_2(x)^{\frac{\ba_3(x)}{r(x)}}\ve{t}^{\ba_3(x)}\ \forall (x,t)\in \overline{\q} \times \R$

Then we have $W_0^{1,\h}(\q)\hookrightarrow L^{\f}(\q)$ and 
\be 
S=\inf_{u \in W_0^{1,\h}(\q)\setminus\{0\}} \frac{\mm{u}}{\mm{u}_{\f}}>0.
\ee
\textbf{Concentration compactness principle}
The proof following theorem follows from \cite[Theorem 2.1]{vdp7}.
    \begin{theorem} \label{thm3}
     Assume that $(H_2)$ and $(H_4)$ hold. Let $\{u_n\}$ be a bounded sequence in $W_0^{1,\h}(\q)$ which satisfies the following 
        \begin{align}
            u_n \rightharpoonup u \ \text{in} \ W_0^{1,\h}(\q),\\
           \label{4} \h(\cdot,\p{u_n}) \xrightharpoonup{*} \mu \ \text{in sense of measures},\\
            \label{3}\f(\cdot,\ve{u_n})\xrightharpoonup{*}\nu \ \text{in sense of measures} 
        \end{align}
        where $\mu$ and $\nu$ are bounded measures on $\q.$ Then, there exists $\{x_i\}_{i\in I} \subset C$ of distinct points and $\{\mu_i\}_{i\in I}$, $\{\nu_i\}_{i\in I}\subset (0,\infty)$; $I$ is at most countable set such that 
        \begin{align}
            \nu= &\f(\cdot,\ve{u})+\sum_{i\in I} \nu_i \delta_{x_i}\\
            \mu \geq & \h(\cdot,\p{u})+\sum_{i\in I} \mu_i \delta_{x_i}\\
           \label{2} S \min\{{\nu_i}^{\frac{1}{p^*(x_i)}},{\nu_i}^{\frac{1}{r^*(x_i)}}\} \leq & \max\{{\mu_i}^{\frac{1}{p(x_i)}},{\mu_i}^{\frac{1}{r(x_i)}}\} \ \forall \ i \in I
        \end{align}
          \end{theorem}
Next, we recall the HLSI for variable exponent Sobolev spaces.
\begin{lemma}\cite{cv1} \label{var}
Assume that \( z, s \in C_{+}(\overline{\Omega}) \), \( m \in L^{z^+}(\Omega) \cap L^{z^{-}}(\Omega) \), \( w \in L^{s^{+}}(\Omega) \cap L^{s^{-}}(\Omega) \) and \( d : \overline{\Omega} \times \overline{\Omega} \to \mathbb{R} \) be a continuous function such that \( 0 < d^{-} \leq d^{+} < N \) and
\[
\frac{1}{z(x)} +\frac{d(x, y)}{N} + \frac{1}{s(y)} = 2.
\]
Then there exists a constant \( c > 0 \), independent of \( m, w \), such that the inequality
\[
\left| \int_{\Omega} \int_{\Omega} \frac{m(x) w(y)}{|x - y|^{d(x, y)}} dx dy \right| \leq c \left( \|m\|_{L^{z^+}(\Omega)} \|w\|_{L^{s^+}(\Omega)} + \|m\|_{L^{z^{-}}(\Omega)} \|w\|_{L^{s^{-}}(\Omega)} \right)
\]
holds. Moreover, if 
\[
\frac{1}{z(x)} +\frac{d(x, y)}{N} + \frac{1}{z(y)} = 2,
\] 
 then 
\[
\left| \int_{\Omega} \int_{\Omega} \frac{w(x) w(y)}{|x - y|^{d(x, y)}} dx dy \right| \leq c \left(\mm{w}^2_{L^{\frac{2N}{2N-d^+}}(\q)}+ \mm{w}^2_{L^{\frac{2N}{2N-d^-}}(\q)}\right).
\]
\end{lemma}

The next lemma is a version of HLSI for the space $W^{1,\h}(\q)$.
\begin{lemma}\label{HL} If $(L_1)$ holds, then for every $u\in W_0^{1,\h}(\q)$ we have 
    %HLSI for the space $W^{1,\h}(\q)$
\be \label{20} \left| \int_{\Omega} \int_{\Omega} \frac{F(x,u(x)) F(y,u(y))}{|x - y|^{d(x, y)}} dx dy \right| \leq C \max \left(\mm{u}^{2\al_1^-}, \ \mm{u}^{2\al_3^+}\right)\ee
\end{lemma}
\begin{proof}
Firstly we will show that $F(\cdot, u(\cdot)) \in L^{\frac{2N}{2N-d^+}}(\q)$.
Consider
\begin{align*}
    &\int_{\q}\ve{F(x,u(x))}^{\frac{2N}{2N-d^+}}\,dx \leq c  \int_{\q}\left(\ve{u}^{\al_1(x)}+a_1(x)^{\frac{\al_2(x)}{q(x)}}\ve{u}^{\al_2(x)}+a_2(x)^{\frac{\al_3(x)}{r(x)}}\ve{u}^{\al_3(x)}\right)^{\frac{2N}{2N-d^+}}\,dx\\
    &\leq c \int_{\q}\left(\ve{u}^{\al_1(x){\frac{2N}{2N-d^+}}}+a_1(x)^{{\frac{2N}{2N-d^+}}\frac{\al_2(x)}{q(x)}}\ve{u}^{{\frac{2N}{2N-d^+}}\al_2(x)}+a_2(x)^{{\frac{2N}{2N-d^+}}\frac{\al_3(x)}{r(x)}}\ve{u}^{{\frac{2N}{2N-d^+}}\al_3(x)}\right)\,dx\\
    &\leq c \max \left(\mm{u}_{\psi}^{\frac{2N\al_1^-}{2N-d^+}},\mm{u}_{\psi}^{\frac{2N\al_3^+}{2N-d^+}}\right)\\
    &\leq c_{\psi} \max \left(\mm{u}^{\frac{2N\al_1^-}{2N-d^+}},\mm{u}^{\frac{2N\al_3^+}{2N-d^+}}\right)<\infty.
    %&\text{Moreover,} \ \mm{F(x,u(x))}_{L^{\frac{2N}{2N-d^+}}(\q)}^2\leq c \max \left(\mm{u}^{2\al_1^-},\mm{u}^{2\al_3^+}\right)
    \end{align*}
where $\psi(x,u)=\ve{u}^{\al_1(x){\frac{2N}{2N-d^+}}}+a_1(x)^{{\frac{2N}{2N-d^+}}\frac{\al_2(x)}{q(x)}}\ve{u}^{{\frac{2N}{2N-d^+}}\al_2(x)}+a_2(x)^{{\frac{2N}{2N-d^+}}\frac{\al_3(x)}{r(x)}}\ve{u}^{{\frac{2N}{2N-d^+}}\al_3(x)}$ and the last two inequalities follow from Lemma \ref{rel} and Lemma \ref{embb}, respectively.
Moreover, 
$$\mm{F(x,u(x))}_{L^{\frac{2N}{2N-d^+}}(\q)}^2\leq c \max \left(\mm{u}^{2\al_1^-},\mm{u}^{2\al_3^+}\right)$$
Similarly, we have 
\be
\mm{F(x,u(x))}_{L^{\frac{2N}{2N-d^-}}(\q)}^2\leq c \max \left(\mm{u}^{2\al_1^-},\mm{u}^{2\al_3^+}\right).
\ee
Thus, we have $F(\cdot, u(\cdot))\in L^{{\frac{2N}{2N-d^-}}}(\q)\cap L^{\frac{2N}{2N-d^+}}(\q)$. From the Lemma \ref{var}, we get 
\begin{multline*}
\left| \int_{\Omega} \int_{\Omega} \frac{F(x,u(x)) F(y,u(y))}{|x - y|^{d(x, y)}} dx dy \right| \leq C\left(\mm{F(\cdot, u(\cdot))}^2_{L^{\frac{2N}{2N-d^-}}(\q)}+\mm{F(\cdot, u(\cdot))}^2_{L^{\frac{2N}{2N-d^+}}(\q)}\right) \\
\leq c \max \left(\mm{u}^{2\al_1^-},\mm{u}^{2\al_3^+}\right).
\end{multline*}
\end{proof}
This Lemma also holds for  $u\in W^{1,\h}(\q)$.
% \begin{remark}
%  The inequality \eqref{20} also holds when $\frac{2N\al_1}{2N-d^+}$     
% \end{remark}

Next, we define the notion of a weak solution of \eqref{model}. A function $u \in W_0^{1,\h}(\q)$ is said to be weak solution for the problem \eqref{model} iff
\begin{align}
    \begin{split}
        M(\varphi_{\h}(\p{u}))\int_{\q}\left(\p{u}^{p(x)-2}+a_1(x)\p{u}^{q(x)-2}+a_2(x)\p{u}^{r(x)-2}\right)\nabla u\nabla v \,dx\\
        =\e \int_{\q}g(x)\ve{u}^{\ga(x)-2}uv\,dx+\theta \int_{\q} B(x,u)v\,dx+\kappa \int_{\Omega} \int_{\Omega} \frac{F(y,u(y))f(x,u(x))v(x)}{\ve{x-y}^{d(x,y)}}\,dxdy, 
    \end{split}
\end{align}
holds for all $v \in W_0^{1,\h}(\q)$.
The corresponding energy functional is given by $J: W_0^{1,\h}(\q) \to \R$
\begin{align}
    \begin{split}
    J(u)&=\Hat{M}(\varphi_{\h}(\p{u}))
    -\theta\int_{\q}\left(c_1(x)\frac{\ve{u(x)}^{\ba_1(x)}}{\ba_1(x)}+c_2(x)a_1(x)^{\frac{\ba_2(x)}{q(x)}}\frac{\ve{u}^{\ba_2(x)}}{\ba_2(x)}+c_3(x)a_2(x))^{\frac{\ba_3(x)}{r(x)}}\frac{\ve{u}^{\ba_3(x)}}{\ba_3(x)}\right)\,dx\\-&\e\int_{\q}\frac{g(x)}{\ga(x)}\ve{u}^{\ga(x)}\, dx-\frac{\kappa}{2}\int_{\Omega} \int_{\Omega} \frac{F(y,u(y))F(x,u(x))}{\ve{x-y}^{d(x,y)}}\,dxdy
    \end{split}
\end{align}
One can verify that $J\in C^1\left(W_0^{1,\h}(\q),\R \right)$ and a critical point of $J$ is a weak solution to \eqref{model}.
\section{Proof of Theorem \ref{thm1}} 
The proof of Theorem \ref{thm1} relies on the Mountain Pass theorem and Ekeland's variational principle.
\begin{lemma} \label{mpg1}
Let $(L_1)-(L_2)$ and $(M_1)-(M_2)$ hold. Then 
    \begin{enumerate}[(i)]
    \item there exist $\omega, \tau,\e_0$ such that for $\e\in(0,\e_0),$  $J(u)\geq \tau$ for every $u\in W_0^{1,\h}(\q)$ with $\mm{u}=\omega$.
    \item for every $\e,\theta,\kappa>0,$ there exists $z\in W_0^{1,\h}(\q)$ such that $J(z)<0,$ provided $\mm{z}>\omega$.
    \end{enumerate}
\end{lemma}
\begin{proof}
$(i)$ By using $(M_2), (M_1)$ Lemma \ref{rel} and Lemma \ref{HL} for $u\in W_0^{1,\h}(\q)$ with $\mm{u}<1,$ we have 
    \begin{align*}
        J(u)&\geq \eta m_0 \varphi_{\h}(\p{u}) -\frac{\e}{\ga^-}\int_{\q}g(x)\ve{u}^{\ga(x)}\,dx-\frac{\theta}{\ba_1^-}\int_{\q}\f(x,u)\,dx\\
        &-\frac{\kappa}{2}\int_{\Omega} \int_{\Omega} \frac{F(y,u(y))F(x,u(x))}{\ve{x-y}^{d(x,y)}}\,dxdy\\
        &\geq \frac{m_0 \eta}{r^+}\x_{\h}(\nabla u)-c\frac{\e}{\ga^-}\mm{u}^{\ga^-}-\frac{\theta}{\ba_1^-}\max\left(\mm{u}^{\ba_1^-}_{\f},\mm{u}^{\ba_3^+}_{\f}\right)
        -\frac{\kappa}{2} c \max \left(\mm{u}^{2\al_1^-},\mm{u}^{2\al_3^+}\right)\\
        & \geq \frac{m_0 \eta}{r^+}\mm{u}^{r^+}-c\frac{\e}{\ga^-}\mm{u}^{\ga^-}-\frac{\theta}{\ba_1^-}c_{\f}\mm{u}^{\ba_1^-}-\kappa C \mm{u}^{2\al_1^-}.
    \end{align*}
    Since $r^+<\ba_1^-$ and $r^+<2\al_1^-$ so, we can choose $\mm{u}=\omega<1$ to be small enough such that 
    \[c_{\omega}=\frac{m_0 \eta}{r^+} \omega^{r^+}-\frac{\theta}{\ba_1^-}c_{\f}\omega^{\ba_1^-}-\kappa C \omega^{2\al_1^-}>0\]
    Choose $\e_0=\frac{c_{\omega}\ga^-}{2c \omega^{\ga^-}}$ then for $\e \in (0, \e_0)$ we get $J(u)\geq \frac{c_{\omega}}{2}>0$. So, $J(u)\geq \tau=\frac{c_{\omega}}{2}$ whenever $\mm{u}=\omega$.

    $(ii)$ By $(M_2)$ we have 
        \be \label{3.1}
        \Hat{M}(t)\leq \frac{\Hat{M}(t_0)}{t_0^{\frac{1}{\eta}}}t^{\frac{1}{\eta}}=c t^{\frac{1}{\eta}} \ \forall t \geq t_0>0
        \ee
        Take $v \in C_{c}^{\infty}(\q)\setminus\{0\}$ with $\mm{v}=1$ then for $t\geq t_0>1$, by using \eqref{3.1} we have 
        \begin{align*}
          J(tv)& \leq c(\varphi_{\h}(\p{t v}))^{\frac{1}{\eta}}-\e \int_{\q}\frac{g(x)}{\ga(x)}\ve{tv}^{\ga(x)}\, dx-\theta \int_{\q}\Hat{B}(x,tv)\,dx\\
          & \leq \frac{c t^{\frac{r^+}{\eta}}}{(p^-)^{\frac{1}{\eta}}}\left(\x_{\h}(\nabla v)\right)^{\frac{1}{\eta}}-\frac{\e t^{\ga^-}}{\ga^+}\int_{\q}g(x)\ve{v}^{\ga(x)}\,dx-\frac{\theta t^{\ba_1^-}}{\ba_3^+}\int_{\q}\f(x,v)\,dx\\
          \implies & J(tv)\to -\infty \ \text{as}\  t \to \infty. 
        \end{align*}
       Thus, we can choose $t_1>1$ and $z=t_1v$ such that $\mm{z}>\omega$ and $J(z)<0$. 
    \end{proof}
    \begin{lemma} \label{m3}
      Assume that $(M_3)$ and $(L_2)$ hold, then there exists $v(>0)\in W_0^{1,\h}(\q)$ such that $J(v)<0$ and $\mm{v}<\omega.$
  \end{lemma}
  \begin{proof}
      Choose $v \in C_{c}^{\infty}(\q)\setminus\{0\}$ with $\mm{v}=1$ and $t \in (0,\omega)$ then by $(M_3)$
  \[
          J(tv)\leq c\frac{t^{\zeta p^-}}{(p^-)^{\zeta}} \left(\x_{\h}(\nabla v)\right)^{\zeta}-\e \frac{t^{\ga^+}}{\ga^+}\int_{\q}g(x)\ve{v}^{\ga(x)}\,dx-\frac{\theta t^{\ba_3^+}}{\ba_3^+}\int_{\q}\f(x,v)\,dx
      \]
      Since $\zeta p^->\ba_3^+>\ga^+$ so for sufficiently small $t_* \in (0,\omega),$ we have $J(t_* v)<0$.
  \end{proof}
  \begin{lemma} \label{PS} Assume that $(H_2)-(H_4), (L_1)-(L_2)$ and $(M_1)-(M_2)$ hold then the functional $J$ satisfies 
    $(PS)_c$ condition for $c<c_*,$ where 
    \be
    c_*=\left(\frac{1}{l}-\frac{1}{\ba_1^-}\right)m_0 \min \left(S^{h_1},S^{h_2}\right) \min\left(\left(\frac{m_0}{\theta} \right)^{b_1},\left(\frac{m_0}{\theta} \right)^{b_2}\right)
    \ee
    where 
    $h_1=\left(\frac{p r^*}{r^*-p}\right)^-, \ h_2=\left(\frac{r p^*}{p^*-r}\right)^+$
    and $b_1=\left(\frac{p}{r^*-p}\right)^-, \ b_2= \left(\frac{r}{p^*-r}\right)^+$.
  \end{lemma}
  \begin{proof}
     Let $\{u_n\}$ be $(PS)_c$ a sequence for $J$. From $(M_2),(M_1),(L_2)$ and Lemma \ref{rel}, we have 
\begin{align*}
      & c+O_n(1)+\mm{u_n} \geq J(u_n)-\frac{1}{l}\langle J'(u_n), u_n \rangle\\
      &=\Hat{M}(\varphi_{\h}(\p{u_n}))-\frac{1}{l} M(\varphi_{\h}(\p{u_n}))\x_{\h}(\nabla u_n)+\e \int_{\q}\left(\frac{1}{l}-\frac{1}{\ga(x)}\right)g(x)\ve{u_n}^{\ga(x)}\,dx\\
     & +\theta\int_{\q}\left(c_1(x)\left(\frac{1}{l}-\frac{1}{\ba_1(x)}\right)\ve{u_n}^{\ba_1(x)}+\left(\frac{1}{l}-\frac{1}{\ba_2(x)}\right)c_2(x)a_1(x)^{\frac{\ba_2(x)}{q(x)}}\ve{u_n}^{\ba_2(x)}\right)\,dx \\
    &+\theta\int_{\q}c_3(x)\left(\frac{1}{l}-\frac{1}{\ba_3(x)}\right)a_2(x))^{\frac{\ba_3(x)}{r(x)}}\ve{u_n}^{\ba_3(x)}\,dx\\
    & \nonumber+\frac{\kappa}{2}\int_{\Omega} \int_{\Omega}\frac{F(y,u_n(y))\left(\frac{2}{l}f(x,u_n(x))u_n(x)-F(x,u_n(x))\right)}{\ve{x-y}^{d(x,y)}}\,dxdy\\
    &\geq\left(\frac{\eta}{r^+}-\frac{1}{l}\right)m_0 \x_{\h}(\nabla u_n)+\theta \left(\frac{1}{l}-\frac{1}{\ba_1^-}\right)\int_{\q}\f(x,\ve{u_n}) \,dx\\
    & \geq  \left(\frac{\eta}{r^+}-\frac{1}{l}\right)m_0 \x_{\h}(\nabla u_n)\geq \left(\frac{\eta}{r^+}-\frac{1}{l}\right)m_0  \min \left(\mm{u_n}^{p^-},\mm{u_n}^{r^+}\right).
    \end{align*}
    So, the sequence $\{u_n\}$ is bounded. From Lemma \ref{emb} we get exist a constant $K \in [1.\infty)$ with 
    \be 
    K:=1+\max \left(\sup_{n\in N}\int_{\q}\h(x,\p{u_n})\,dx,\  \sup_{n\in N}\int_{\q}\f(x,\ve{u_n})\,dx \right)
    \ee
    As a consequence of Theorem \ref{thm3} we have 
    \begin{align}\label{5}
    \begin{cases}
       u_n \rightharpoonup u \ \text{in} \ W_0^{1,\h}(\q),\\
       u_n \to u \ a.e.\ \text{in} \ \q \\ 
           \h(\cdot,\p{u_n}) \xrightharpoonup{*} \mu \geq \h(\cdot,\p{u})+\sum_{i\in I} \mu_i \delta_{x_i}\\
\f(\cdot,\ve{u_n})\xrightharpoonup{*}\nu=\f(\cdot,\ve{u})+\sum_{i\in I} \nu_i \delta_{x_i} \ \\
            S \min\{{\nu_i}^{\frac{1}{p^*(x_i)}},{\nu_i}^{\frac{1}{r^*(x_i)}}\} \leq  \max\{{\mu_i}^{\frac{1}{p(x_i)}},{\mu_i}^{\frac{1}{r(x_i)}}\} \\
            \varphi_{\h}(\p{u_n})\to \chi  
    \end{cases}
\end{align}
with $\chi>0$ is a constant.
Next, we will show that $I=\emptyset$. Let, if possible, $I \neq \emptyset$ say $i \in I$. Take $v\in C_c^{\infty}(\q)$ such that $0 \leq v\leq 1$ with $v\equiv 1$ in $B_{1/2}$ and $v \equiv 0$ in $B^c_{1}$. For any given $\ep>0$ and $i\in I$, define 
 \[v_{i,\ep}(x)=v\left(\frac{x-x_i}{\ep}\right).\]
 For fixed $\ep>0$ the sequence $\{v_{i,\ep}u_n\}$ is bounded so,
$\langle J'(u_n), v_{i,\ep}u_n\rangle=O_n(1)$
Thus, 
\begin{align} \label{3.3}
    \begin{split}
 M(\varphi_{\h}(\p{u_n}))\int_{\q}v_{i,\ep}(x)\h(x,\ve{\nabla u_n})\,dx \leq O_n(1)+\e \int_{\q} g(x)\ve{u_n}^{\ga(x)} v_{i,\ep}(x)\,dx\\
 +\theta \int_{\q}\f(x,u_n)v_{i,\ep}(x)\,dx+\kappa \int_{\Omega} \int_{\Omega}\frac{F(y,u_n(y))f(x,u_n(x))u_n(x)v_{i,\ep}(x)}{\ve{x-y}^{d(x,y)}}\,dxdy\\
 +M(\varphi_{\h}(\p{u_n})) \left \lvert \int_{\q}\left(\p{u_n}^{p(x)-2}+a_1(x)\p{u_n}^{q(x)-2}+a_2(x)\p{u_n}^{r(x)-2}\right)u_n\nabla u_n\nabla v_{i,\ep} \,dx \right\rvert
    \end{split}
\end{align}
Let $\delta>0$ be arbitrary, then by Young's inequality, we have 
 \begin{align*}
   &\left \lvert \int_{\q}\left(\p{u_n}^{p(x)-2}+a_1(x)\p{u_n}^{q(x)-2}+a_2(x)\p{u_n}^{r(x)-2}\right)u_n\nabla u_n\nabla v_{i,\ep} \,dx \right\rvert\\
&\leq \int_{\q}\left(\p{u_n}^{p(x)-1}\ve{u_n \nabla v_{i,\ep}}+a_1(x)\p{u_n}^{q(x)-2}\ve{u_n \nabla v_{i,\ep}}+a_2(x)\p{u_n}^{r(x)-2}\ve{u_n \nabla v_{i,\ep}}\right)\,dx\\
&\leq \delta \int_{\q}\h(x,\ve{\nabla u_n})\,dx+C_{\delta}\int_{\q}\h(x,\ve{u_n \nabla v_{i,\ep}})\,dx\\
&\leq K\delta+C_{\delta}\int_{\q}\h(x,\ve{u_n \nabla v_{i,\ep}})\,dx
 \end{align*}
 \begin{align} \label{3.4}
     \begin{split}
\left \lvert\int_{\q}\left(\p{u_n}^{p(x)-2}+a_1(x)\p{u_n}^{q(x)-2}+a_2(x)\p{u_n}^{r(x)-2}\right)u_n\nabla u_n\nabla v_{i,\ep} \,dx \right\rvert \leq \left(K\delta+C_{\delta}\int_{\q}\h(x,\ve{u_n \nabla v_{i,\ep}})\,dx\right)    
     \end{split}
 \end{align}
 Since $u_n \to u \ \text{in} \ L^{\h}(\q)$,
 \be
 \lim_{n \to \infty}\int_{\q}\h(x,\ve{u_n \nabla v_{i,\ep}})\,dx=\int_{\q}\h(x,\ve{u \nabla v_{i,\ep}})\,dx
 \ee
 By H\"older's inequality, we have 
 \begin{align*}
    \int_{\q}\h(x,\ve{u \nabla v_{i,\ep}})\,dx&=\int_{B_{i,\ep}}\left(\ve{u \nabla v_{i,\ep}}^{p(x)}+a_1(x)\ve{u \nabla v_{i,\ep}}^{q(x)}+a_2(x)\ve{u \nabla v_{i,\ep}}^{r(x)}\right)\,dx\\
    &\leq 2\mm{\ve{u}^{p(\cdot)}}_{L^{\frac{p^*(\cdot)}{p(\cdot)}}(B_{i,\epsilon})} \mm{\ve{\nabla v_{i,\epsilon}}^{p(\cdot)}}_{L^{\frac{N}{p(\cdot)}}(B_{i,\epsilon})}\\
 &+2\mm{a_1(\cdot)\ve{u}^{q(x)}}_{L^{\frac{q^*(\cdot)}{q(\cdot)}}(B_{i,\epsilon})} \mm{\ve{\nabla v_{i,\epsilon}}^{q(\cdot)}}_{L^{\frac{N}{q(\cdot)}}(B_{i,\epsilon})}\\
 &+2\mm{a_2(\cdot)\ve{u}^{r(x)}}_{L^{\frac{r^*(\cdot)}{r(\cdot)}}(B_{i,\epsilon})} \mm{\ve{\nabla v_{i,\epsilon}}^{r(\cdot)}}_{L^{\frac{N}{r(\cdot)}}(B_{i,\epsilon})}.
  \end{align*}
 Now, by Lemma \ref{rel}, we have \[\mm{\ve{\nabla v_{i,\epsilon}}^{p(\cdot)}}_{L^{\frac{N}{p(\cdot)}}(B_{i,\epsilon})}\leq \max\left(\left(\int_{B_{i,\ep}} \ve{\nabla v_{i,\epsilon}(x)}^{{p(\cdot)}\frac{N}{p(\cdot)}}\,dx \right)^\frac{{p^-}}{N},\left(\int_{B_{i,\ep}} \ve{\nabla v_{i,\epsilon}(x)}^{{p(\cdot)}\frac{N}{p(\cdot)}}\,dx\right)^\frac{{p^+}}{N}\right)\]
 $\leq \max \left(\left(\int_{B_1}\p{v}^N\,dy\right)^{\frac{{p^-}}{N}},\left(\int_{B_1}\p{v}^N\,dy\right)^{\frac{{p^+}}{N}}\right)<\infty.$
% with $y=\frac{x-x_i}{\ep}$
Similarly,
\[\mm{\ve{\nabla v_{i,\epsilon}}^{q(\cdot)}}_{L^{\frac{N}{q(\cdot)}}(B_{i,\epsilon})}< \infty \ \text{and}\ \mm{\ve{\nabla v_{i,\epsilon}}^{r(\cdot)}}_{L^{\frac{N}{r(\cdot)}}(B_{i,\epsilon})}<\infty.\]
So, \be \label{3.6}
\lim_{\ep \to 0}\int_{\q}\h(x,\ve{u \nabla v_{i,\ep}})\,dx=0.
\ee
Since $\delta>0$ is arbitrary, from \eqref{3.4} and \eqref{3.6}, we get 
\be \label{3.7}
\lim_{\ep \to 0} \lim_{n \to \infty}M(\varphi_{\h}(\p{u_n})) \left \lvert \int_{\q}\left(\p{u_n}^{p(x)-2}+a_1(x)\p{u_n}^{q(x)-2}+a_2(x)\p{u_n}^{r(x)-2}\right)u_n\nabla u_n\nabla v_{i,\ep} \,dx \right\rvert=0.
\ee
From Lemma \ref{g} and \eqref{5}, we have 
\be
\lim_{n \to \infty}\int_{\q} g(x)\ve{u_n}^{\ga(x)}v_{i,\ep}(x)\,dx=\int_{\q} g(x) \ve{u}^{\ga(x)}v_{i,\ep}(x)\,dx.
\ee
Moreover, one can obtain 
\be \label{8}
\lim_{\ep \to 0}\int_{\q} g(x) \ve{u}^{\ga(x)}v_{i,\ep}(x)\,dx =0.
\ee
Next, we will prove that
\be \label{3.8}
\lim_{n \to \infty}\int_{\Omega} \int_{\Omega}\frac{F(y,u_n(y))f(x,u_n(x))u_n(x)v_{i,\ep}(x)}{\ve{x-y}^{d(x,y)}}\,dx\,dy=\int_{\Omega} \int_{\Omega}\frac{F(y,u(y))f(x,u(x))u(x)v_{i,\ep}(x)}{\ve{x-y}^{d(x,y)}}\,dx \,dy.
\ee
Thanks to $(L_1)$, $F(\cdot,u_n(\cdot))\in L^{{\frac{2N}{2N-d^-}}}(\q)\cap L^{\frac{2N}{2N-d^+}}(\q)$
and $u_n \to u \ a.e.  \ \text{in}\ \q$. By continuity of $F,$ we get 
$$ F(\cdot,u_n(\cdot))\to F(\cdot,u(\cdot)) \ a.e. \  \text{in}\ \q.$$
Define operator $\K:L^{\frac{2N}{2N-d^+}}(\q) \to L^{\frac{2N}{d^+}}(\q)$ as $\K(w(x))=\frac{1}{\ve{x}^{d^+}}*w(x)$.

Then, $\K$ is a bounded linear operator.
Since $F(\cdot,u_n(\cdot))$ is bounded in $L^{\frac{2N}{2N-d^+}}(\q),$ $\K\left(F(\cdot,u_n(\cdot))\right)$ is a bounded sequence in $L^{\frac{2N}{d^+}}(\q)$. Thus, up to a subsequence 
\[\K\left(F(\cdot,u_n(\cdot))\right) \rightharpoonup \K\left(F(\cdot,u(\cdot))\right) \ \text{in}\ L^{\frac{2N}{d^+}}(\q).\]
So,
\be \label{3.9}
 \ \int_{\q}\phi(x)\K\left(F(\cdot,u_n(\cdot))\right) \,dx \to  \int_{\q}\phi(x)\K\left(F(\cdot,u(\cdot))\right) \,dx \ \forall \phi \in L^{\frac{2N}{2N-d^+}}(\q).
\ee
We have,
\begin{align*}
    \int_{\q} \ve{f(x,u)u v_{i,\ep}}^{\frac{2N}{2N-d^+}}\,dx \leq \int_{\q}c\left(\ve{u}^{\al_1(x)}+a_1(x)^{\frac{\al_2(x)}{q(x)}}\ve{u}^{\al_2(x)}+a_2(x)^{\frac{\al_3(x)}{r(x)}}\ve{u}^{\al_3(x)}\right)^{\frac{2N}{2N-d^+}}\,dx\\
    \leq c \int_{\q}\left(\ve{u}^{\al_1(x){\frac{2N}{2N-d^+}}}+a_1(x)^{{\frac{2N}{2N-d^+}}\frac{\al_2(x)}{q(x)}}\ve{u}^{{\frac{2N}{2N-d^+}}\al_2(x)}+a_2(x)^{{\frac{2N}{2N-d^+}}\frac{\al_3(x)}{r(x)}}\ve{u}^{{\frac{2N}{2N-d^+}}\al_3(x)}\right)\,dx< \infty.
\end{align*}
Thus, from \eqref{3.9}, we get 
\be \label{3.10}
\lim_{n \to \infty}\int_{\Omega} \int_{\Omega}\frac{F(y,u_n(y))f(x,u(x))u(x)v_{i,\ep}(x)}{\ve{x-y}^{d^+}}\,dxdy=\int_{\Omega} \int_{\Omega}\frac{F(y,u(y))f(x,u(x))u(x)v_{i,\ep}(x)}{\ve{x-y}^{d^+}}\,dx\,dy.
\ee
Next, our claim is 
\be \label{3.11}
\lim_{n \to\infty}\int_{\q} \left(\int_{\q}\frac{F(y,u_n(y))}{\ve{x-y}^{d^+}}\,dy\right)\left(f(x,u_n(x))u_n(x)-f(x,u(x))u(x)\right)v_{i,\ep}(x)\,dx=0.\ee
This claim will follow from H\"older's inequality and boundedness of sequence $\{u_n\},$ if we show that 
\[\mm{\left(f(x,u_n(x))u_n(x)-f(x,u(x))u(x)\right)v_{i,\ep}}_{L^{\frac{2N}{2N-d^+}}(\q)} \to 0 \ \text{as}\ n \to \infty.\]
The above expression follows from the dominated convergence theorem,  \eqref{5} and assumption $(L_1)$. Hence, \eqref{3.11} holds.
Now, from \eqref{3.10} and \eqref{3.11}, we have
\be \label{3.12}
\lim_{n \to \infty}\int_{\Omega} \int_{\Omega}\frac{F(y,u_n(y))f(x,u_n(x))u_n(x)v_{i,\ep}}{\ve{x-y}^{d^+}}\,dxdy=\int_{\Omega} \int_{\Omega}\frac{F(y,u(y))f(x,u(x))u(x)v_{i,\ep}}{\ve{x-y}^{d^+}}\,dx\,dy.
\ee
Following similar steps, one can show that \eqref{3.12} still holds if we replace $d^+$ with $d^-$.
Hence, \eqref{3.8} follows.
Next, we will show that 
\be\label{3.13}
\lim_{\ep \to 0}\int_{\Omega} \int_{\Omega}\frac{F(y,u(y))f(x,u(x))u(x)v_{i,\ep}(x)}{\ve{x-y}^{d(x,y)}}\,dxdy=0.
\ee
By Lemma \ref{HL}, we have 
\begin{multline} 
\int_{\Omega} \int_{\Omega}\frac{F(y,u(y))f(x,u(x))u(x)v_{i,\ep}}{\ve{x-y}^{d(x,y)}}\,dx\,dy \leq c \mm{F(y,u(y))}_{L^{\frac{2N}{2N-d^+}}(\q)}\mm{f(x,u(x))u(x)v_{i,\ep}}_{L^{\frac{2N}{2N-d^+}}(\q)}\\
+c\mm{F(y,u(y))}_{L^{\frac{2N}{2N-d^-}}(\q)}\mm{f(x,u(x))u(x)v_{i,\ep}}_{L^{\frac{2N}{2N-d^-}}(\q)}
\end{multline}
As a consequence of the dominated convergence theorem, we have 
\[\mm{f(x,u(x))u(x)v_{i,\ep}}_{L^{\frac{2N}{2N-d^-}}(\q)}\to 0 \ \text{as} \ \ep \to 0\] and 
\[\mm{f(x,u(x))u(x)v_{i,\ep}}_{L^{\frac{2N}{2N-d^+}}(\q)}\to 0 \ \text{as} \ \ep \to 0.\]
Thus, \eqref{3.13} follows.
Hence, 
\be \label{3.15}
\lim_{\ep \to 0}\lim_{n \to \infty}\int_{\Omega} \int_{\Omega}\frac{F(y,u_n(y))f(x,u_n(x))u_n(x)v_{i,\ep}}{\ve{x-y}^{d(x,y)}}\,dxdy=0.
\ee
By using \eqref{3.7}, \eqref{3.15}, \eqref{8}, $(M_1)$ and passing limits ${\ep \to 0}$  and $ {n \to \infty}$ in \eqref{3.3}, we get 
\be \label{6}
m_0 \mu_i \leq \theta \nu_i.
\ee
From \eqref{5} and \eqref{6}, we have 
\begin{align*}
   S \min\left(\left(\frac{m_0}{\theta}\mu_i\right)^{\frac{1}{p^*(x_i)}},\left(\frac{m_0}{\theta}\mu_i\right)^{\frac{1}{r^*(x_i)}}\right)&\leq \max \left(\mu_i^{\frac{1}{p(x_i)}},\mu_i^{\frac{1}{r(x_i)} }\right)\\
    S \left(\frac{m_0}{\theta}\mu_i\right)^{\frac{1}{h_i}} &\leq \mu_{i}^{\frac{1}{b_i}}\\
    \text{where} \ h_i\in\{p^*(x_i),r^*(x_i)\}\ \text{and}\ b_i \in\{p(x_i),r(x_i)\}\\
    \implies \mu_i \geq S^{\frac{h_i b_i}{h_i-b_i}}\left(\frac{m_0}{\theta}\right)^{\frac{b_i}{h_i-b_i}} \\
    \text{We have,}\ h_1 \leq \frac{p(x_i)r^*(x_i)}{r^*(x_i)-p(x_i)}\leq \frac{h_ib_i}{h_i-b_i}\leq \frac{p^*(x_i)r(x_i)}{p^*(x_i)-r(x_i)}\leq h_2\\
    b_1 \leq \frac{p(x_i)}{r^*(x_i)-p(x_i)}\leq \frac{b_i}{h_i-b_i}\leq\frac{r(x_i)}{p^*(x_i)-r(x_i)}\leq b_2\\
    \end{align*}
    Thus, 
    \be
    \mu_i \geq \min \left(S^{h_1},S^{h_2}\right) \min\left(\left(\frac{m_0}{\theta}\right)^{b_1},\left(\frac{m_0}{\theta}\right)^{b_2}\right).
    \ee
   By using $(M_2),(M_1),(L_2)$ and \eqref{5}, we have 
   $$ \begin{aligned}
  & c=\lim_{n \to \infty} \left(J(u_n)-\frac{1}{l}\langle J'(u_n),u_n\right)\\
  &=\lim _{n \rightarrow \infty}\left[\Hat{M}(\varphi_{\h}(\p{u_n}))-\frac{1}{l} M(\varphi_{\h}(\p{u_n}))\x_{\h}(\nabla u_n)+\e \int_{\q}\left(\frac{1}{l}-\frac{1}{\ga(x)}\right)g(x)\ve{u_n}^{\ga(x)}\,dx\right.\\
    &\left.+\theta\int_{\q}\left(c_1(x)\left(\frac{1}{l}-\frac{1}{\ba_1(x)}\right)\ve{u_n}^{\ba_1(x)}+\left(\frac{1}{l}-\frac{1}{\ba_2(x)}\right)c_2(x)a_1(x)^{\frac{\ba_2(x)}{q(x)}}\ve{u_n}^{\ba_2(x)}\right)\,dx \right.\\
    &\left.+\theta\int_{\q}c_3(x)\left(\frac{1}{l}-\frac{1}{\ba_3(x)}\right)a_2(x))^{\frac{\ba_3(x)}{r(x)}}\ve{u_n}^{\ba_3(x)}\,dx\right.\\
      &\left.+\frac{\kappa}{2}\int_{\Omega} \int_{\Omega}\frac{F(y,u_n(y))\left(\frac{2}{l}f(x,u_n(x))u_n(x)-F(x,u_n(x))\right)}{\ve{x-y}^{d(x,y)}}\,dxdy\right]\\
& \geq \lim_{n \rightarrow \infty}\left[\left(\frac{\eta}{r^+}-\frac{1}{l}\right)m_0 \x_{\h}(\nabla u_n)+\theta \left(\frac{1}{l}-\frac{1}{\ba_1^-}\right)\int_{\q}\f(x,u_n) \,dx\right] \\
& \geq \lim_{n \rightarrow \infty} \theta\left(\frac{1}{l}-\frac{1}{\ba_1^-}\right)\int_{\q} v_{i,\ep} \f(x,u_n) \,dx\\
&=\left(\frac{1}{l}-\frac{1}{\ba_1^-}\right)\theta\int_{\q} v_{i,\ep} \, d\nu_n\\
&\geq \left(\frac{1}{l}-\frac{1}{\ba_1^-}\right)\theta \nu_i\\
& \geq \left(\frac{1}{l}-\frac{1}{\ba_1^-}\right) m_0\min \left(S^{h_1},S^{h_2}\right) \min\left(\left(\frac{m_0}{\theta}\right)^{b_1},\left(\frac{m_0}{\theta}\right)^{b_2}\right),
\end{aligned}
$$
 which is a contradiction. So, $I=\emptyset.$
    Thus, from \eqref{5} we get 
    \be
    \lim_{n \to \infty}\int_{\q}\f(x,u_n)\,dx=\int_{\q} \f(x,u)\,dx.
    \ee
    As a consequence of Brezis Lieb's lemma, we get 
    $\lim _{n \to \infty}\int_{\q}\f(x,u_n-u)\,dx=0.$
    From Lemma \ref{rel}, we have $\lim_{n \to \infty} \mm{u_n-u}_{\f}=0$.
    
   As a consequence of  H\"older's inequality and $\lim_{n \to \infty} \mm{u_n-u}_{\f}=0$, one can easily obtain
   \be \label{21}
   \lim_{n \to \infty} \int_{\q}B(x,u_n)(u_n-u)\,dx=0
   \ee
 Also, by H\"older's inequality,$(H_3)$ 
    \[\int_{\q}g(x)\ve{u_n}^{\ga(x)-2}u_n(u_n-u)\,dx\leq 2 \mm{g(x)^{\frac{\ga(x)-1}{\ga(x)}}\ve{u_n}^{\ga(x)-1}}_{L^{\frac{\ga(x)}{\ga(x)-1}}(\q)}\mm{g(x)^{\frac{1}{\ga(x)}}(u_n-u)}_{L^{\ga(x)}(\q)} \to 0 \ \text{as}\ n \to \infty\]
Next, we have
\begin{multline} \label{3.19}
\left \lvert \int_{\Omega}\frac{F(y,u_n(y))f(x,u_n(x))(u_n(x)-u(x))}{\ve{x-y}^{d(x,y)}}\,dxdy\right \rvert \leq c\mm{F(y,u_n(y))}_{L^{\frac{2N}{2N-d^+}}(\q)}\\
\mm{f(x,u_n(x))(u_n(x)-u(x))}_{L^{\frac{2N}{2N-d^+}}(\q)}+c\mm{F(y,u_n(y))}_{L^{\frac{2N}{2N-d^-}}(\q)}\mm{f(x,u_n(x))(u_n(x)-u(x))}_{L^{\frac{2N}{2N-d^-}}(\q)}.  
\end{multline}
Now, by assumption $(L_1)$, H\"older's inequality  and Lemma \ref{embb}, we get 
\begin{multline}
    \int_{\q}\ve{f(x,u_n)(u_n-u)}^{\frac{2N}{2N-d^+}}\,dx \leq 2c\mm{\ve{u_n}^{(\al_1(x)-1)\frac{2N}{2N-d^+}}}_{L^{\frac{\al_1(x)}{\al_1(x)-1}}(\q)} \mm{\ve{u_n-u}^{\frac{2N}{2N-d^+}}}_{L^{\al_1(x)}(\q)}
    \\+2c\mm{a_1(x)^{\frac{2N(\al_2(x)-1)}{q(x)(2N-d^+)}}\ve{u_n}^{(\al_2(x)-1)\frac{2N}{2N-d^+}}}_{L^{\frac{\al_2(x)}{\al_2(x)-1}}(\q)} \mm{a_1(x)^{\frac{2N}{q(x)(2N-d^+)}}\ve{u_n-u}^{\frac{2N}{2N-d^+}}}_{L^{\al_2(x)}(\q)}\\
    +2c\mm{a_2(x)^{\frac{2N(\al_3(x)-1)}{q(x)(2N-d^+)}}\ve{u_n}^{(\al_3(x)-1)\frac{2N}{2N-d^+}}}_{L^{\frac{\al_3(x)}{\al_3(x)-1}}(\q)} \mm{a_2(x)^{\frac{2N}{q(x)(2N-d^+)}}\ve{u_n-u}^{\frac{2N}{2N-d^+}}}_{L^{\al_3(x)}(\q)}
    \to 0 \ \text{as} \ n \to \infty.
    \end{multline}
  Thus, 
  \[\lim_{n \to \infty}\mm{f(x,u_n(x))(u_n(x)-u(x))}_{L^{\frac{2N}{2N-d^+}}(\q)}=0\]
  and following similar steps, we can get 
  \[\lim_{n \to \infty}\mm{f(x,u_n(x))(u_n(x)-u(x))}_{L^{\frac{2N}{2N-d^-}}(\q)}=0.\] 
  So, from \eqref{3.19}, we have
\be \label{22}
\lim_{n \to \infty}  \int_{\q} \int_{\q}\frac{F(y,u_n(y))f(x,u_n(x))(u_n(x)-u(x))}{\ve{x-y}^{d(x,y)}}\,dx\,dy =0.
\ee
Since $\langle J'(u_n), (u_n-u)\rangle=O_n(1),$
\begin{multline}
    M(\varphi_{\h}(\p{u_n}) \int_{\q}\left(\p{u_n}^{p(x)-2}+a_1(x)\p{u_n}^{q(x)-2}+a_2(x)\p{u_n}^{r(x)-2}\right)\nabla u_n \nabla(u_n-u)\,dx\\
    =\langle J'(u_n), (u_n-u)\rangle+\e \int_{\q}g(x)\ve{u_n}^{\ga(x-2)}u_n(u_n-u)\,dx\\
    +\theta \int_{\q} B(x,u_n)(u_n-u)\,dx+\kappa\int_{\q} \int_{\q}\frac{F(y,u_n(y))f(x,u_n(x))(u_n(x)-u(x)}{\ve{x-y}^{d(x,y)}}\,dxdy+O_n(1).
\end{multline}
Passing limit ${n \to \infty}$ and using \eqref{5}, $ M(\varphi_{\h}(\p{u_n})\to M(\chi) \geq m_0$, \eqref{21},\eqref{22} we get 
$$\lim_{n \to \infty} \int_{\q}\left(\p{u_n}^{p(x)-2}+a_1(x)\p{u_n}^{q(x)-2}+a_2(x)\p{u_n}^{r(x)-2}\right)\nabla u_n \nabla(u_n-u)\,dx=0.$$
Thanks to Lemma \ref{S}, up to a subsequence, we get 
$u_n \to u$ in $W_0^{1,\h}(\q)$.
  \end{proof}
\begin{lemma}
    There exists $\kappa^*>0$ such that for every $\kappa>\kappa^*,$ we have 
    \[c_{\kappa}=\inf_{\gamma \in \Gamma} \max_{0 \leq t \leq 1}J(\gamma(t))<c_*,\]
    where $\Gamma=\{\ga \in C([0,1],W_0^{1,\h}(\q)); \ga(0)=0,\ga(1)=e,J(e)<0\}$.
\end{lemma}
\begin{proof}
Let $0<z\in W_0^{1,\h}(\q)$ with $\mm{z}=1$.  As $J(0)=0$ and by Lemma \ref{mpg1} $J$ has mountain pass geometry so there exists $t_{\kappa}>0$ such that $J(t_{\kappa}z)=\max_{t \geq 0}J(tz)$.
So, $\langle J'(t_{\kappa}z),t_{\kappa}z\rangle=0$ i.e. 
\begin{multline}\label{3.23}
   M(\varphi_{\h}(t_{\kappa}\p{z}))\x_{\h}(t_{\kappa}\nabla z)=\e \int_{\q}t_{\kappa}^{\ga(x)} g(x) \ve{z}^{\ga(x)}\,dx\\
   +\theta \int_{\q} \f(x,t_{\kappa}z)\,dx+\kappa \int_{\q} \int_{\q}\frac{F(y,t_{\kappa}z(y))f(xt_{\kappa}z(x))t_{\kappa}z(x)}{\ve{x-y}^{d(x,y)}}\,dxdy
\end{multline}
Let if possible $t_{\kappa}>\max(t_0,1)$ then by \eqref{3.23}, $(M_2),$ we have 
\begin{align*}
    & \e t_{\kappa}^{\ga^-} \int_{\q} g(x)\ve{z}^{\ga(x)}\,dx+\theta t_{\kappa}^{\ba_1^-}\int_{\q}\f(x,z)\,dx \leq \e \int_{\q}t_{\kappa}^{\ga(x)} g(x) \ve{z}^{\ga(x)}\,dx+\theta \int_{\q} \f(x,t_{\kappa}z)\,dx
    \\
    &+\kappa \int_{\q} \int_{\q}\frac{F(y,t_{\kappa}z(y))f(xt_{\kappa}z(x))t_{\kappa}z(x)}{\ve{x-y}^{d(x,y)}}\,dxdy\\
    &= M(\varphi_{\h}(t_{\kappa}\p{z}))\x_{\h}(t_{\kappa}\nabla z) \leq \frac{r^+}{\eta} \Hat{M}(\varphi_{\h}(t_{\kappa}\p{z}))\\
    &\leq c\frac{r^+}{\eta} (\varphi_{\h}(t_{\kappa}\p{z}))^{\frac{1}{\eta}} \leq c\frac{r^+}{\eta} t_{\kappa}^{\frac{r^+}{\eta}}(\varphi_{\h}(\p{z}))^{\frac{1}{\eta}}.
\end{align*}
Since $\frac{r^+}{\eta}<\ga^-<\ba_1^-$, the sequence $\{t_{\kappa}\}$ is bounded.
Hence there exist a subsequence $\kappa_i \to \infty$ as $i \to \infty$ and $t_* \geq 0$ such that $t_{\kappa_i} \to t_*$ as $i \to \infty$.
Next, we claim that $t_*=0.$ Let if possible $t_*>0.$ Then from $(L_1)$ and the dominated convergence theorem, 
\[\int_{\q} \int_{\q}\frac{F(y,t_{\kappa_i}z(y))f(x,t_{\kappa_i}z(x))t_{\kappa_i}z(x)}{\ve{x-y}^{d(x,y)}}\,dxdy \to \int_{\q} \int_{\q}\frac{F(y,t_{*}z(y))f(x,t_{*}z(x))t_{*}z(x)}{\ve{x-y}^{d(x,y)}}\,dxdy \]
$\implies \lim_{i \to \infty}\kappa_i \int_{\q} \int_{\q}\frac{F(y,t_{*}z(y))f(x,t_{*}z(x))t_{*}z(x)}{\ve{x-y}^{d(x,y)}}\,dxdy=\infty. $

By \eqref{3.23}, we get 
\[ M(\varphi_{\h}(t_{*}\p{z}))\x_{\h}(t_{*}\nabla z)=\infty,\]
which is a contradiction, since $M$ is continuous and $t_*$ is bounded.
%, hence $ M(\varphi_{\h}(t_{*}\p{z}))\x_{\h}(t_{*}\nabla z)< \infty$.
Thus $t_*=0,$ i.e., $t_{\kappa} \to 0$ as $\kappa \to \infty.$
Choose $e=T_0z$ such that $J(T_0z)<0$. For $\ga(t)=tT_0z,$ we have
\[c_{\kappa}=\inf_{\gamma \in \Gamma} \max_{0 \leq t \leq 1}J(\gamma(t)) \leq\max_{0 \leq t \leq 1}J(tT_0z)\leq J(t_{\kappa} z)\leq \Hat{M}(\varphi_{\h}((t_k \p{z}))).\]
Since $M$ is continuous and $t_{\kappa} \to 0$ as $\kappa \to \infty,$ we have 
\[
\lim_{\kappa \to \infty} \Hat{M}(\varphi_{\h}((t_k \p{z})))=0 \implies \lim_{\kappa \to \infty}c_{\kappa}=0.
\]
Thus, there exist $\kappa^*>0$ such that for all $\kappa>\kappa^*,$  $c_{\kappa}<c_*$.
\end{proof}
\begin{proof}[Proof of Theorem \ref{thm1}]
 We have $J(0)=0$, and from Lemma \ref{mpg1}, $J$ satisfies mountain pass geometry. Furthermore, by Lemma \ref{PS}, $J$ satisfies $(PS)_c$ condition. Thus by the mountain pass theorem,  there exists $u \in W_0^{1,\h}(\q)$  such that $J(u)=c_{\kappa}>0$ and $J'(u)=0$. Consequently, $u$ is a nontrivial weak solution of \eqref{model}.

 By Lemma \ref{mpg1}, $J$ is bounded below on $\overline{B}_{\omega}(0)$. From Lemma \ref{m3}, we have 
 \[\bar{c_{\kappa}}=\inf_{u \in \overline{B}_{\omega}(0)}J(u)<0<\inf_{u \in \partial{B}_{\omega}(0)}J(u).\]
 Choose $\ep>0$ such that 
 \[0<\ep<\inf_{u \in \partial{B}_{\omega}(0)}J(u)-\inf_{u \in \overline{B}_{\omega}(0)} J(u).\]
 Thanks to Ekeland's variational principle, there exists a sequence $\{u_{\ep}\} \subset B_{\omega}$ such that 
 \begin{gather}
 \begin{split}
   \bar{c_{\kappa}} <J(u_{\ep})<\bar{c_{\kappa}}+\ep\\
J(u_{\ep})\leq J(u)+\ep \mm{u-u_{\ep}},\ u\neq u_{\ep}  
 \end{split}
\end{gather}
 $\forall u \in \overline{B}_{\omega}(0)$. Applying  Lemma \ref{PS}, there exists $v \in W^{1,\h}_0(\q)$ such that $J(v)=\bar{c_{\kappa}}<0$ and $J'(v)=0$. Hence, $v$ is another solution of \eqref{model}, since $J(u)>0$.
\end{proof}
% \begin{theorem}
%     Let $(M_1)-(M_2)$, $(L_1)-(G_3)$ and $(H_3)$ hold, then there exists a constant $\e_*>0$ and a sequence $\{\theta_n\}\subset (0,\infty)$ with $\theta_n> \theta_{n+1}$ such that the problem has at least n pair of nontrivial solutions for $0<\e\leq \e_*$ and $\theta_{n+1}<\theta\leq \theta_{n}$.
% \end{theorem}
\section{Proof of Theorem \ref{thm2}}
The genus theory and symmetric mountain pass theorem are the primary tools to establish Theorem \ref{thm2}.

For a Banach space $X$, define 
\[\Theta=\{U\in X \setminus \{0\}:U \ \text{is closed in }\  X \ \text{and symmetric w..r.t. origin}\}\]
\[\ga(U)=\text{genus of } \ U \in \Theta\]
\[S_c=\{u\in X:J(u)=c, J'(u)=0\}.\]
\begin{lemma}\cite{Minimax}\label{fin}
    Let $X$ be an infinite dimensional Banach space, and $J \in C^1(X, \mathbb{R})$ be a functional satisfying the $(P S)_c$ condition for $0<c<c_*, J(u)=J(-u)$, and $J(0)=0$. If $X=Y \oplus Z$, where $Y$ is finite dimensional, and J satisfies:
    \begin{itemize}
        \item [$(S_1)$] there exist constants $\rho, \alpha>0$ such that $J(u) \geq \alpha$ for all $u \in \partial B_\rho \cap Z$;
        \item [$(S_2)$] for any finite dimensional subspace $\widetilde{X} \subset X$ there exists $R=R(\widetilde{X})>0$ such that $J(u) \leq 0$ for all $u \in \widetilde{X} \backslash B_{R(\widetilde{X})}$, where $B_{R(\widetilde{X})}=\{x \in \widetilde{X}:\|u\|<R\}$.
    \end{itemize}
Let $Y$ be $t$ dimensional and $Y=\operatorname{span}\left\{e_1, \ldots, e_t\right\}$. For $n \geq t$, inductively choose $e_{n+1} \notin X_n:=$ span $\left\{e_1, \ldots, e_n\right\}$. We assume $R_n=R\left(X_n\right)$ and $D_n=B_{R_n} \cap X_n$. Define
$$
G_n:=\left\{z \in C\left(D_n, X\right): z \text { is odd and } z(u)=u, \forall u \in \partial B_{R_n} \cap X_n\right\}
$$
and
$$
\Gamma_s:=\left\{z \overline{\left(D_n \backslash V\right)}: z \in G_n, n \geq s, V \in \Theta, \text { and } \gamma(V) \leq n-s\right\}
$$
where $\gamma(V)$ is genus of $V$. For each $s \in \mathbb{N}$, let
$$
c_s=\inf _{T \in \Gamma_s} \max_{u \in T} J(u).
$$

Then, $0<\alpha \leq c_s \leq c_{s+1}$ for $s>t$, and if $c_s<c_*$, we have that $c_s$ is the critical value of J. Moreover, for $s>t$, if $c_s=c_{s+1}=\cdots=c_{s+\tau}=c<c_*$, then $\gamma\left(S_c\right) \geq \tau+1$.
\end{lemma}
\begin{lemma}
There is a sequence $\left\{D_n\right\} \subset(0,+\infty)$ independent of $\theta$ with $D_n \leq D_{n+1}$, such that for any $\lambda, \theta, \kappa>0$,
$$
c_n^{\theta}=\inf_{T \in \Gamma_n} \max_{u \in T} J(u)<D_n .
$$
\end{lemma}
\begin{proof}
    \begin{align*}
    & c_n^{\theta}=\inf_{T \in \Gamma_n} \max _{u \in T} J(u)\\
     &\leq \inf_{T \in \Gamma_n} \max _{u \in T} \left(\Hat{M}(\varphi_{\h}(\p{u}))
    -\e\int_{\q}\frac{g(x)}{\ga(x)}\ve{u}^{\ga(x)}\, dx-\frac{\kappa}{2}\int_{\Omega} \int_{\Omega} \frac{F(y,u(y))F(x,u(x))}{\ve{x-y}^{d(x,y)}}\,dxdy \right)<D_n,\\
    &\text{where}\ D_n=\inf_{T \in \Gamma_n} \max _{u \in T} \left(\Hat{M}(\varphi_{\h}(\p{u}))
    -\e\int_{\q}\frac{g(x)}{\ga(x)}\ve{u}^{\ga(x)}\, dx-\frac{\kappa}{2}\int_{\Omega} \int_{\Omega} \frac{F(y,u(y))F(x,u(x))}{\ve{x-y}^{d(x,y)}}\,dxdy +1\right).
    \end{align*}
    Thus, $D_n<\infty$ and $D_n \leq D_{n+1}$ follows from the definition of $\Gamma_n$.
\end{proof}
\textbf{Proof of Theorem \ref{thm2}}
From $(L_3),$ we know that $J$ is an even functional with $J(0)=0$. From Lemma \ref{PS}, $J$ satisfies $(PS)_c$ condition for $c<c_*$.
The proof of $(S_1)$ in Lemma \ref{fin} is similar to the proof of Lemma \ref{mpg1}. Now, we need to prove $(S_2)$.

Let $\widetilde{X}\subset W_0^{1,\h}(\q)$ be a fixed finite-dimensional space. Let us assume $R=R(\widetilde{X})>1,$ then for $u \in \widetilde{X}$ with $\mm{u}\geq R>1,$ we have 
 \begin{multline} \label{3.25}
    J(u) \leq c \left(\varphi_{\h}(\p{u}\right)^{\frac{1}{\eta}}-\e \int_{\q}\frac{g(x)}{\ga(x)}\ve{u}^{\ga(x)}\,dx-\theta \int_{\q} \Hat{B}(x,u)\,dx\\
   \leq \frac{c}{{p^-}^{\frac{1}{\eta}}}\left(\x_{\h}(\p{u}\right)^{\frac{1}{\eta}}-\frac{\e}{\ga^+}\int_{\q}g(x)\ve{u}^{\ga(x)}\,dx-\frac{\theta}{\ba_3^+}\int_{\q} \f(x,u)\,dx.
  \end{multline}
Since $\widetilde{X}$ is a finite dimensional space so, all the norms are equivalent.  
\[\int_{\q}g(x)\ve{u}^{\ga(x)}\,dx \geq c \mm{u}^{\ga^-}; \int_{\q} \f(x,u)\,dx \geq c \mm{u}^{\ba_1^-} \ \forall u \in \widetilde{X}.\]
From \eqref{3.25} and Lemma \ref{rel}, we get 
\[J(u)\leq \frac{c}{(p^-)^{\frac{1}{\eta}}} \mm{u}^{\frac{r^+}{\eta}}-\frac{\e c}{\ga^+}\mm{u}^{\ga^-}-\frac{\theta c}{\ba_3^+}\mm{u}^{\ba_1^-}.\]
 As ${\frac{r^+}{\eta}}<\ba_1^-$ so we can choose $R>0$ large enough such that $J(u)<0$ whenever $\mm{u}\geq R$. Thus, $(S_2)$ follows.

 Choose $\theta_n>0$ such that 
 \[D_n < \left(\frac{1}{l}-\frac{1}{\ba_1^-}\right)m_0 \min \left(S^{h_1},S^{h_2}\right) \min\left(\left(\frac{m_0}{\theta_n} \right)^{b_1},\left(\frac{m_0}{\theta_n} \right)^{b_2}\right)\]
 Thus, for $\theta_{n+1}\leq\theta< \theta_n$, we have 
 \[0 \leq c_1^{\theta}\leq c_2^{\theta}\leq \cdots \leq c_n^{\theta}<D_n<c_*\]
Hence, by Lemma \ref{fin},  $c_1^\theta \leq c_2^\theta \leq \cdots \leq c_n^\theta$ are critical points of $J$. Thus, if $c_1^\theta<c_2^\theta<\cdots<c_n^\theta$, then the problem \eqref{model} has at least $n$ pairs of nontrivial solutions, and if $c_s^\theta=c_{s+1}^\theta$ for some $s=1,2, \ldots, n$, it follows from Lemma \ref{fin} that $S_{c_s^\theta}$ is an infinite set, and problem \eqref{model} has infinitely many nontrivial solutions. \qed
\subsection*{Acknowledgment} The first author acknowledges the financial support from the University Grants Commission, India (Ref. No. 191620115768). The second author is supported by the Science and Engineering Research Board, India, under the CRG/2020/002087 grant.

\end{document}